\documentclass[11pt]{amsart}
\usepackage{amsmath,amssymb,amsfonts,graphics}
\usepackage{dsfont}
\usepackage{textcomp}
\newtheorem{thm}{Theorem}[section]
\newtheorem{cor}[thm]{Corollary}
\newtheorem{lem}[thm]{Lemma}
\newtheorem{prop}[thm]{Proposition}
\theoremstyle{definition}
\newtheorem{defn}[thm]{Definition}
\theoremstyle{remark}
\newtheorem{exm}[thm]{Example}
\newtheorem{rem}[thm]{Remark}
\numberwithin{equation}{section}

\newcommand{\ds}{\displaystyle}
\newcommand{\A}{\mathcal{A}}
\newcommand{\Au}{\widetilde{\mathcal{A}}}
\newcommand{\B}{\mathcal{B}}
\newcommand{\eps}{\epsilon}
\newcommand{\R}{\mathbb{R}}
\newcommand{\N}{\mathbb{N}}

\newcommand{\C}{\mathbb{C}}

\newcommand{\bg}{B(L^2(G))}

\newcommand{\om}{\omega}

\newcommand{\sg} {\sigma}

\begin{document}

\title[Norm-controlled inversion in weighted convolution algebras]
{Norm-controlled inversion in weighted convolution algebras}

\author{Ebrahim Samei}
\address{Department of Mathematics and Statistics, University of Saskatchewan, Saskatoon, Saskatchewan, S7N 5E6, Canada}
\email{samei@math.usask.ca}

\author{Varvara Shepelska}
\address{Department of Mathematics and Statistics, University of Saskatchewan, Saskatoon, Saskatchewan, S7N 5E6, Canada}
\email{shepelska@gmail.com}

\footnote{{\it Date}: \today.



The first named author was partially supported by NSERC Grant no. 409364-2015. The second named author was partially supported by a PIMS Postdoctoral Fellowship at the University of Saskatchewan.}

\maketitle

\begin{abstract}
Let $G$ be a discrete group, let $p\ge1$, and let $\omega$ be a weight on $G$. Using the approach from~\cite{GK}, we provide sufficient conditions on a weight $\omega$ for $\ell^p(G,\omega)$ to be a Banach algebra admitting a norm-controlled inversion in the reduced C$^*$-algebra of $G$, namely $C^*_r(G)$. We show that our results can be applied to various cases including locally finite groups as well as finitely generated groups of polynomial or intermediate growth and a natural class of weights on them. These weights are of the form of polynomial or certain subexponential functions. We also consider the non-discrete case and study the existence of norm-controlled inversion in $\bg$ for some related convolution algebras.
\end{abstract}

\section{Introduction}\label{S:Introduction}

The study of the phenomenon of norm-controlled inversion was initiated by Nikolski in \cite{Nikolski}. Suppose that $\A$ is a commutative unital Banach algebra continuously embedded into the space $C(X)$ of continuous functions on a Hausdorff topological space $X$. For $0<\delta\le1$ he defines the majorant $c_1(\delta,\A,X)$ by
$$
c_1(\delta,\A,X)=\sup\{\|f^{-1}\|_{\A}:\, f\in\A,\,\delta\le|f(x)|\le\|f\|_{\A}\le1,\,x\in X\},
$$
where $\|f^{-1}\|_{\A}$ is assumed to be $\infty$ if $f$ is not invertible in $\A$. Nikolski says that $\delta$ allows the norm control of the inverse in $\A$ (with respect to $X$) if $c_1(\delta,\A,X)<\infty$. He further defines a critical constant $\delta_1(\A,X)$ by
$$
\delta_1(\A,X)=\inf\{0<\delta\le1:\,c_1(\delta,\A,X)<\infty\}.
$$
Then the equality $\delta_1(\A,X)=0$ is, obviously, equivalent to the existence of a norm-controlled inversion in $\A$ for every $\delta$. One of the main results in \cite{Nikolski} provides the estimates for $\delta_1(\A,X)$ and $c_1(\delta,\A,X)$ in case when $\A$  is $\ell^1(G)$ for a discrete abelian group $G$ or is the unitization $L^1(G)+\C\cdot e$ of the group algebra of a locally compact abelian non-discrete group $G$, and $X$ is the dual group $\hat{G}$. In particular, it is shown that in these cases $\delta_1(\A,\hat{G})\ge1/2$ if $G$ is infinite implying the lack of universal, i.e. independent of $\delta$, norm-controlled inversion. On the other hand, as was shown in \cite{E-FNZ}, there will be a universal norm-controlled inversion for certain weighted group algebras $\ell^p(\mathbb{Z},\omega)$.

In \cite{GK}, Gr\"{o}chenig and Klotz considered a phenomenon of norm-controlled inversion in more general settings. Let $\A\subseteq\B$ be two Banach algebras with a common unit. Recall that $\A$ is said to be inverse-closed in $\B$ if for every $a\in A$ the existence of $a^{-1}$ in $\B$ implies that $a^{-1}\in\A$. In this case, we say that $\A$ {\it admits norm-controlled inversion in $\B$} if there is a function $h:\R^+\times \R^+\to \R^+$ such that
\[\label{eq:norm control}
\|a^{-1}\|_{\A}\le h(\|a\|_{\A},\|a^{-1}\|_{\B}).
\]
Because of the nature of the algebra $C(X)$, for $f\in C(X)$ the relation $|f(x)|\ge\delta$ is equivalent to $\|f^{-1}\|_{C(X)}\le1/\delta$. Hence, we can rewrite the definition of Nikolski's majorant $c_1(\delta,\A,X)$ as
$$
c_1(\delta,\A,X)=\sup\{\|f^{-1}\|_{\A}:\, f\in\A,\,\|f\|_{\A}\le1,\,\|f^{-1}\|_{C(X)}\le1/\delta\}.
$$
It is then easy to see that $\A$ admits norm-controlled inversion in $C(X)$ in the sense of Gr\"{o}chenig and Klotz if and only if $\delta_1(\A,X)=0$, i.e. $\A$ admits universal norm-controlled inversion in the sense of Nikolski. It is proved in \cite{GK} that if $\B$ is a $C^*$-algebra and $\A$ is a differential *-subalgebra of $\B$, then $\A$ admits norm-controlled inversion in $\B$.

The study of a general norm-controlled inversion was inspired by results in applied mathematics and non-commutative geometry where inversion preserves the smoothness of elements in certain Banach algebras and the desire to control the smoothness norm of the inverses. For example, it is known that if an infinite matrix possesses certain off-diagonal decay and is invertible as an operator on $\ell^2$, then the inverse matrix has the same kind of off-diagonal decay (see \cite{Baskakov}, \cite{GL}, \cite{Jaffard}). Refining the Jaffard's theorem from \cite{Jaffard}, Gr\"{o}chenig and Klotz prove in \cite{GK2} that the subalgebra of $B(\ell^2)$ of infinite matrices with off-diagonal decay of order $r>1$ admits norm-controlled inversion in $B(\ell^2)$.

In this paper, we use the approach from \cite{GK} to continue the exploration of Nikolski and study the phenomenon of norm-controlled inversion for non-commutative weighted group algebras.

The paper is organized as follows. In Section \ref{S:Norm control-diff subalg}, we present the proof of the existence of norm-controlled inversion in a subalgebra with a modified differential norm inside a $C^*$-algebra. Section \ref{S:Diff subalg-weight Lp spaces} is devoted to building a technical base for proving that the algebras of interest are modified differential subalgebras in the corresponding C$^*$-algebras. In Section \ref{S:Norm control-weighted Lp sapce-discrete groups} we focus on the discrete case. Here we provide sufficient conditions on a weight $\omega$ for $\ell^p(G,\omega)$ to admit a norm-controlled inversion in $C^*_r(G)$ and apply these results to finitely generated groups of polynomial growth or intermediate growth and a natural class of weights on them, including polynomial and certain subexponential weights. In Section \ref{S:Norm control-weighted Lp sapce-Non discrete groups}, we show that in the non-discrete case the same conditions are sufficient for the unitization of an algebra $L^p(G,\omega)\cap L^2(G)$ to admit a norm-controlled inversion in $\bg$. Finally, in the Appendix we present the proof of a technical result from Section \ref{S:Norm control-diff subalg} providing an asymptotic form of a norm-controlling function.

\section{Norm-controlled inversion for subalgebras with a modified differential norm in $C^*$-algebras}\label{S:Norm control-diff subalg}

As was already mentioned in Section \ref{S:Introduction}, one of the main results in \cite{GK} asserts the existence of norm-controlled inversion in a differential *-subalgebra of a $C^*$-algebra.

\begin{thm}[{\cite[Theorem~1.1(i)]{GK}}]\label{original main result}
Let $\B$ be a $C^*$-algebra and $\A\subset\B$ be a Banach *-algebra with the same unit. Assume that $\A$ is a differential *-subalgebra of $\B$, i.e. there is $C>0$ such that
\begin{align}\label{differential norm-general}
\|ab\|_{\A}\le C(\|a\|_{\A}\|b\|_{\B}+\|b\|_{\A}\|a\|_{\B}),\quad a,b\in\A.
\end{align}
Then $\A$ admits norm-controlled inversion in $\B$.
\end{thm}

The precise formula for the controlling function $h$ is given in \cite[Theorem~3.3]{GK}. It was also mentioned in \cite[Section 5]{GK} that the method presented can be adopted to prove an analogue of Theorem~\ref{original main result} for subalgebras with a {\em modified differential norm} satisfying
\begin{align}\label{power inequality}
\|a^2\|_A\le C\|a\|_{A}^{1+\theta}\|a\|_{\B}^{1-\theta},\quad a\in\A,
\end{align}
where constant $C>0$ and exponent $0<\theta<1$ are fixed. In fact, the authors noted that their proof is a modification of the one of \cite[Theorem~1.1]{Sun}, where a specific case of $\B=B(\ell^2(\mathbb{Z}^d))$ and its subalgebra $Q_{p,\alpha}$ of matrices with polynomial off-diagonal decay that had a modified differential norm with respect to $B(\ell^2(\mathbb{Z}^d))$ was considered.

For the sake of completeness and for the purpose of further modification in Section \ref{S:Norm control-weighted Lp sapce-Non discrete groups}, we present a proof of the analogue of \cite[Theorem~3.3]{GK} for modified differential norm which is an adaptation of the corresponding proof from \cite{GK}.

\begin{prop}\label{norm control general}
Let $\mathcal{B}$ be a $C^*$-algebra and $\mathcal{A}\subset\mathcal{B}$ be a Banach $*$-algebra with the same unit satisfying~\eqref{power inequality} for some $C>0$ and $0<\theta<1$. Then $\A$ is inverse-closed in $\B$ and whenever $a\in \A$ is invertible, we have
\begin{equation}\label{norm estimate}
\|a^{-1}\|_{\A}\le\frac{\|a\|_{\A}}{\|a\|_{\B}^2}\,\prod_{k=0}^{\infty}\left(1+\left(2\frac{\|a\|_{\A}^2}{\|a\|_{\B}^2}\right)^ {(1+\theta)^k}C^{\frac{(1+\theta)^k-1}{\theta}}\left(1-\frac{1}{\|a\|_{\B}^2\|a^{-1}\|_{\B}^2}\right)^{2^k-(1+\theta)^k}\right).
\end{equation}
\end{prop}
%
\begin{proof}
Let $c\in\A$ and $n\in\mathbb{N}$. Substituting $a=c^n$ into \eqref{power inequality} we obtain
$$
\|c^{2n}\|_{\A}\le C\|c^n\|_{\A}^{1+\theta}\|c^n\|_{\B}^{1-\theta}.
$$
Taking the $n$-th root and $\lim_{n\to\infty}$ of both sides, we get the following inequality of the spectral radii $\rho_{\A}(c)$ and $\rho_{\B}(c)$
$$
(\rho_{\A}(c))^2\le(\rho_{\A}(c))^{1+\theta}(\rho_{\B}(c))^{1-\theta},
$$
which implies that $\rho_{\A}(c)\le\rho_{\B}(c)$. Since the reverse inequality follows from the inclusion $\A\subset\B$, we have that $\rho_{\A}(c)=\rho_{\B}(c)$ for every $c\in\A$. It follows from \cite[Lemma 3.1]{FGL1} that
$\A$ is inverse-closed in $\B$.

We now prove the norm estimate \eqref{norm estimate}. Applying \eqref{power inequality} to $a=c^{2^{k-1}}$, $c\in\A$, $k\in\mathbb{N}$, we get
\begin{align}\label{eq:estimate}
\|c^{2^k}\|_{\A}\le C\|c^{2^{k-1}}\|_{\A}^{1+\theta}\|c^{2^{k-1}}\|_{\B}^{1-\theta}\le C\|c^{2^{k-1}}\|_{\A}^{1+\theta}\left(\|c\|_{\B}^{2^{k-1}}\right)^{1-\theta}.
\end{align}
If we let $\beta_n=\frac{\|c^n\|_{\A}}{\|c\|_{\B}^n}$, $n\in\mathbb{N}$, then \eqref{eq:estimate} implies that
$$
\beta_{2^k}\le C\beta_{2^{k-1}}^{1+\theta},\quad k\in\mathbb{N}.
$$
Using induction in $k$, we obtain
$$
\beta_{2^k}\le C^\frac{(1+\theta)^k-1}{\theta}\beta_1^{(1+\theta)^k},
$$
and, going back to $c$, we get
\begin{equation}
\|c^{2^k}\|_{\A}\le\left(\frac{\|c\|_{\A}}{\|c\|_{\B}}\right)^{(1+\theta)^k}C^\frac{(1+\theta)^k-1}{\theta}\|c\|_{\B}^{2^k},\quad k\in\N.
\end{equation}
Now take $n\in\N\cup\{0\}$ and consider its dyadic expansion $n=\sum\limits_{k=0}^{\infty} \epsilon_k2^k$. Then
$$
\|c^n\|_{\A}=\left\|\prod_{k=0}^{\infty} (c^{2^k})^{\eps_k}\right\|_{\A}\le \prod_{k=0}^{\infty} \|c^{2^k}\|_{\A}^{\eps_k}\le \prod_{k=0}^{\infty}\left(\left(\frac{\|c\|_{\A}}{\|c\|_{\B}}\right)^{(1+\theta)^k}C^\frac{(1+\theta)^k-1}{\theta}\|c\|_{\B}^{2^k}\right)^{\eps_k}.
$$
Hence, if $\mathcal{F}$ denotes the set of all sequences $\eps=\{\eps_k\}\in\{0,1\}^{\N}$ that contain finitely many $1$-s, then
\begin{eqnarray}\label{series estimate}
\sum\limits_{n=0}^{\infty} \|c^n\|_{\A}&\le \sum\limits_{\eps\in\mathcal{F}} \prod\limits_{k=0}^{\infty}\left(\left(\frac{\|c\|_{\A}}{\|c\|_{\B}}\right)^{(1+\theta)^k} C^\frac{(1+\theta)^k-1}{\theta}\|c\|_{\B}^{2^k}\right)^{\eps_k} \\\nonumber&=\prod\limits_{k=0}^{\infty}\left(1+\left(\frac{\|c\|_{\A}}{\|c\|_{\B}}\right)^{(1+\theta)^k} C^\frac{(1+\theta)^k-1}{\theta}\|c\|_{\B}^{2^k}\right).
\end{eqnarray}
The last infinite product is convergent if and only if
$$
\sum_{k=0}^\infty \left(\frac{\|c\|_{\A}}{\|c\|_{\B}}\right)^{(1+\theta)^k} C^\frac{(1+\theta)^k-1}{\theta}\|c\|_{\B}^{2^k}<\infty,
$$
and since $\theta<1$, it is easy to see that this happens exactly when $\|c\|_{\B}<1$.

Now assume that $a\in\A$ is invertible in $\B$ and set $b = a^*a/\|a^*a\|_{\B}$. Then $b$ is hermitian,
invertible, $\|b\|_{\B} = 1$, and the spectrum $\sigma_{\B}(b)$ is contained in $(0, 1]$. Hence, for the spectrum of an element $c=e-b$ we have $\sigma_{\B}(c)\subseteq[0,1-\epsilon]\subseteq[0,1)$, and, in particular, $\|c\|_{\B}=1-\epsilon<1$. This implies that
$$
b^{-1}=\sum\limits_{n=0}^{\infty}(e-b)^n=\sum\limits_{n=0}^{\infty}c^n,
$$
with convergence in $\B$. Then since
$$
a^{-1}=\frac{b^{-1}a^*}{\|a^*a\|_{\B}},
$$
we can use the above representation of $b^{-1}$ together with \eqref{series estimate} to obtain the following.
\begin{align}\label{first norm estimate}
\|a^{-1}\|_{\A}&\le\frac{\|a^*\|_{\A}}{\|a^*a\|_{\B}}\left(\sum\limits_{n=0}^{\infty} \|c^n\|_{\A}\right)\\\nonumber &\le\frac{\|a^*\|_{\A}}{\|a^*a\|_{\B}}\prod\limits_{k=0}^{\infty}\left(1+\left(\frac{\|c\|_{\A}}{\|c\|_{\B}}\right)^{(1+\theta)^k} (2C)^\frac{(1+\theta)^k-1}{\theta}\|c\|_{\B}^{2^k}\right).
\end{align}
Finally, we estimate the norms of $c = e - a^*a/\|a^*a\|_{\B}$ in $\A$ and $\B$ directly by the norms of $a$ and $a^{-1}$.
First,
\begin{align}\label{cA}
\|c\|_{\A}\le1+\frac{\|a^*a\|_{\A}}{\|a^*a\|_{\B}}\le2\frac{\|a\|_{\A}^2}{\|a\|_{\B}^2}.
\end{align}
On the other hand, since $\B$ is a $C^*$-algebra and $a^*a$ is positive, we have  $$\|(a^*a)^{-1}\|_{\B}^{-1}=\min\{\lambda:\lambda\in\sigma(a^*a)\}=\lambda_{\mathrm{min}},$$ and so
\begin{align}\label{cB}
\|c\|_{\B}=1-\frac{\lambda_{\mathrm{min}}}{\|a^*a\|_{\B}}=1-\frac{1}{\|(a^*a)^{-1}\|_{\B}\|a^*a\|_{\B}} =1-\frac1{\|a^{-1}\|_{\B}^2\|a\|_{\B}^2}.
\end{align}
Combining \eqref{cA} and \eqref{cB} with \eqref{first norm estimate}, we precisely obtain \eqref{norm estimate}.
\end{proof}

As in \cite{GK}, we can further work with \eqref{norm estimate} to obtain a simpler norm-controlling function and study its asymptotic behavior.

\begin{prop}\label{formula}
Let $\mathcal{B}$ be a $C^*$-algebra and $\mathcal{A}\subset\mathcal{B}$ be a Banach $*$-algebra with the same unit satisfying~\eqref{power inequality} for some $C>0$ and $0<\theta<1$. For every invertible $a\in \A$, denote
$$
\nu(a)=\|a\|_{\A}\|a^{-1}\|_{\B},
$$
and let $\gamma=\log_2{(1+\theta)}$.
Then there exist constants $C_1$ and $C_2$ (depending on $\theta$) such that whenever $\nu(a)\ge2$, we have that
$$
\|a^{-1}\|_{\A}\le C_1\|a\|_{\A}\|a^{-1}\|_{\B}^2\,e^{C_2\nu(a)^{\frac{2\gamma}{1-\gamma}}(\ln\nu(a))^{\frac{2-\gamma}{1-\gamma}}}.
$$
\end{prop}

Because the proof is rather technical and is far from the main subject of the present paper, we present it in the Appendix.

\section{Modified differential norm for a certain subspace of $L^p(G,\omega)$ in $C^*_r(G)$}\label{S:Diff subalg-weight Lp spaces}

Let $G$ be a locally compact group. A measurable function $\om:G\to [0,\infty)$  is called a {\it weight} on $G$, if $\om(e)=1$, $\om(x)=\om(x^{-1})$ for all $x\in G$, and $\om$ is submultiplicative, i.e.
$$
\om(xy)\le\om(x)\om(y),\quad x,y\in G.
$$
Let $1\le p<\infty$. As usual, the weighted Banach space $L^p(G,\om)$ is defined by
$$
L^p(G,\om)=\left\{f\,:\,f\om\in L^p(G)\ \text{and}\ \|f\|_{p,\om}=\|f\om\|_p\right\}.
$$
It is known that $L^p(G,\om)$ becomes a Banach algebra with respect to the convolution under certain assumptions on the weight $\om$, see \cite{K1} and \cite{OS1}. We are going to use the following weighted modification of a particular case of \cite[Theorem~3.3]{OS1}.
\begin{thm}\label{Lp algebra}
Let $G$ be a locally compact group, let $1\le p<\infty$, and let $\om$ be a weight on $G$. Furthermore, let $1<q\le\infty$ be the index conjugate to $p$, i.e. $1/p+1/q=1$, and suppose that there exists a function $u\in L^q(G)$ such that
$$
\frac{\om(xy)}{\om(x)\om(y)}\le u(x)+u(y),\quad x,y\in G.
$$
Then $L^p(G,\om)$ is a Banach algebra with respect to the convolution product. Moreover, if we let $\sigma=\om u$, then $L^p(G,\om)\subseteq L^1(G,\sigma)$ and
\begin{align}\label{differential norm1}
 \|f\ast g\|_{p,\om}\le\|f\|_{1,\sigma}\|g\|_{p,\om}+\|f\|_{p,\om}\|g\|_{1,\sigma},\quad f,g\in L^p(G,\om).
\end{align}
\end{thm}


In the case when the conditions of Theorem~\ref{Lp algebra} are satisfied and $\sigma$ is bounded away from zero, we will have the following inclusions
$$
L^p(G,\om)\subseteq L^1(G,\sigma)\subseteq L^1(G)\subset C^*_r(G),
$$
where $C^*_r(G)$ stands for the reduced $C^*$-algebra of $G$.
This enables us to raise the question about the existence of norm-controlled inversion of $L^p(G,\om)$ in $C^*_r(G)$. Keeping in mind the result of Proposition \ref{norm control general}, we seek to find conditions under which a modified differential norm relation \eqref{power inequality} holds for $\A=L^p(G,\om)$ and $\B=C^*_r(G)$. The following theorem provides us with such conditions and in Section \ref{S:Norm control-weighted Lp sapce-discrete groups} we will demonstrate how it can be applied to various cases to obtain norm-controlled inversion of $\ell^p(G,\om)$ in $C^*_r(G)$.



\begin{thm}\label{T:diff norm-weighted Lp}
Let $G$ be a locally compact group, let $\om$ be a weight on $G$, and let $1\le p<\infty$ and $q$ the index conjugate to $p$. Suppose that there exists a bounded measurable function $u:G\to\R^+$,  $s<q$, and $r>0$ such that $\sigma=\om u$ is bounded away from zero,
\begin{align}\label{algebra condition}
\frac{\om(xy)}{\om(x)\om(y)}\le u(x)+u(y),\quad x,y\in G,\quad\text{and}
\end{align}
\begin{align}\label{u_integral condition}
\int\limits_G u(x)^s\om(x)^{r}dx<\infty.
\end{align}
Then there exist $0<\theta<1$ and $C>0$ such that
for every $f\in L^p(G,\om)\cap L^2(G)$
\begin{align}\label{Eq:weighted Lp diff formula}
\|f\ast f\|_{p,\omega}\le C\|f\|_{p,\omega}^{1+\theta}\|f\|_2^{1-\theta}.
\end{align}
\end{thm}

\begin{proof}
From \eqref{differential norm1} it follows that
\begin{align}\label{differential norm1-app}
\|f\ast f\|_{p,\om}\le 2\|f\|_{p,\om}\|f\|_{1,\sigma},\quad f\in L^p(G,\om).
\end{align}
Hence it suffices to show that there exist $0<\theta<1$ and $C>0$ such that
\begin{align}\label{Eq:general pytlik}
\|f\|_{1,\sigma}\le C\|f\|_2^{1-\theta}\|f\|_{p,\om}^{\theta},\quad f\in L^p(G,\om)\cap L^2(G).
\end{align}
For a fixed $f\in L^p(G,\om)\cap L^2(G)$ and $0<\theta<1$ we have
\begin{align*}
\|f\|_{1,\sigma}&=\int\limits_G\,f(x)\sigma(x)\,dx=\int\limits_G\,f(x)\om(x)u(x)\,dx =\int\limits_G\,f(x)^{1-\theta}\big(f(x)\om(x)\big)^{\theta} \big(\om(x)^{1-\theta}u(x)\big)\,dx.
\end{align*}
We now want to apply the generalized H\"{o}lder's inequality with exponents $\frac2{1-\theta}$, $\frac{p}{\theta}$, and $\alpha>1$, determined from the relation
$$
\frac{1-\theta}2+\frac{\theta}{p}+\frac1{\alpha}=1,
$$
to the functions $f(x)^{1-\theta}$, $\big(f(x)\om(x)\big)^{\theta}$, and $\om(x)^{1-\theta}u(x)$. Simple calculations show that $$\alpha=\ds\frac{2p}{p+p\theta-2\theta}\to\frac{p}{p-1}=q \ \ \text{as}\ \  \theta\to1,$$ and we obtain
\begin{align*}
\|f\|_{1,\sigma}\le\left(\int\limits_G\, f(x)^2\,dx\right)^{\frac{1-\theta}2}\left(\int\limits_G\, \big(f(x)\om(x)\big)^p\,dx\right)^{\frac{\theta}p}\left(\int\limits_G\, \big(\om(x)^{1-\theta}u(x)\big)^\alpha\,dx\right)^{\frac1{\alpha}}.
\end{align*}
Hence, to prove \eqref{Eq:general pytlik}, it would be enough to choose $\theta$ so that
$$
\left(\int\limits_G\, \big(\om(x)^{1-\theta}u(x)\big)^\alpha\,dx\right)^{\frac1{\alpha}}= \left(\int\limits_G\, u(x)^\alpha\om(x)^{(1-\theta)\alpha}\,dx\right)^{\frac1{\alpha}} =C<\infty.
$$
Since $\alpha\to q$ as $\theta\to1$ and $s<q$, we can choose $\theta$ so that $s<\alpha$ and $(1-\theta)\alpha<r$. Then the convergence of the above integral will follow directly from \eqref{u_integral condition} since $\om\ge1$ and $u$ is bounded.
\end{proof}

In the rest of this section, we will discuss several situations when Theorem \ref{T:diff norm-weighted Lp} can be applied. The first one is when the weight $\om$ is weakly subadditive, i.e. there exists $D>0$ such that
$$
\om(xy)\le D(\om(x)+\om(y)),\quad x,y\in G.
$$
In this case,
$$
\frac{\om(xy)}{\om(x)\om(y)}\le\frac D{\om(x)}+\frac D{\om(y)},\quad x,y\in G,
$$
and so we can take $u=D/\om$. Condition \eqref{u_integral condition} then becomes
$$
\int\limits_G \frac{D^s}{\om(x)^{s-r}}dx<\infty
$$
for some $s<q$ and $r>0$, which is equivalent to $\om^{-1}\in L^{\tilde{s}}(G)$ for some $\tilde{s}<q$. Hence we obtain the following corollary of Theorem \ref{T:diff norm-weighted Lp}. Note that for the case when $p=1$, the result was obtained in \cite[Lemmas 1 and 2]{Pytlik} but the method used there is different from ours.

\begin{cor}\label{norm control weakly subadditive}
Let $G$ be a locally compact group, let $\omega$ be a weakly subadditive weight on $G$, and let $1\le p<\infty$  and $q$ the index conjugate to $p$. Suppose that $\omega^{-1}\in L^s(G)$ for some $s<q$. Then there are $0<\theta <1$ and $C>0$ such that for every $f\in L^p(G,\om)\cap L^2(G)$, the relation \eqref{Eq:weighted Lp diff formula} holds.
\end{cor}


\begin{exm}\label{E:locally finite-diff Lp norm relation}
Let $G$ be a locally compact group for which there is an increasing sequence  $\{G_i\}_{i\in \N}$ of compact subgroups of $G$ such that $G:=\cup_{i\in \N} G_i$.
Take an increasing sequence $\{n_i\}_{i\in \N}\in [1,\infty)$. Define
$\om:G \to [1,\infty)$ by
\begin{align}\label{Eq:weight-locally finite}
\om=1+\sum_{i=1} n_i 1_{G_{i+1}\setminus G_{i}}
\end{align}
It is easy to see that
$$\om(st)=\max\{ \om(s),\om(t) \},\quad s,t\in G.$$
This, in particular, implies that $\om$ is a weakly subadditive weight on $G$. Moreover, for any $s>0$, we can pick $\{n_i\}$ in such a way that $\om^{-1} \in L^s(G)$. Thus Corollary \ref{norm control weakly subadditive} can be applied to these classes of weights.
\end{exm}

Now we turn to the case of a compactly generated group $G$. Let $U$ be a compact symmetric generating neighborhood of the identity in $G$.
We define the {\it length function} $\tau_U : G \to [0, \infty)$ by
\begin{align*}
\tau_U(x)=\inf \{n\in \N\, :\, x\in U^n \} \ \ \text{for} \ \ x \neq e, \ \ \tau_U(e)=0.
\end{align*}
When there is no fear of ambiguity, we write $\tau$ instead of $\tau_U$. This function can be used to construct many classes of weights on~$G$. In fact, if $\rho:\N\cup \{0\}\to \R^+$ is an increasing concave function with $\rho(0)=0$ and $\rho(n)\to \infty$ as $n\to\infty$, then
\begin{align}\label{Eq:weight-lenght function}
\om(x)=e^{\rho(\tau(x))}, \quad x\in G,
\end{align}
is a weight on $G$. For instance, for every $0< \alpha < 1$, $\beta > 0$, $\gamma >0$, and $C>0$,
we can define the {\it polynomial weight} $\om_\beta$ on $G$ of order $\beta$ by
\begin{align}\label{Eq:poly weight-defn}
\om_\beta(x)=(1+\tau(x))^\beta, \quad x\in G,
\end{align}
and the {\it subexponential weights} $\sg_{\alpha, C}$ and $\nu_{\beta,C}$ on $G$ by
\begin{align}\label{Eq:Expo weight-defn}
\sg_{\alpha,C}(x)=e^{C\tau(x)^\alpha}, \quad x\in G,
\end{align}
\begin{align}\label{Eq:Expo weight II-defn}
\nu_{\gamma,C}(x)=e^\frac{C\tau(x)}{(\ln (1+\tau(x)))^\gamma}, \quad x\in G.
\end{align}

We are particularly interested in compactly generated groups of polynomial or intermediate (subexponential) growths.
Recall that $G$ has {\it polynomial growth} if there exist $C>0$ and $d\in \N$ such that for every $n\in \N$
	$$\lambda(U^n)\leq Cn^d, \quad n\in \N,$$
where $\lambda$ is the Haar measure on $G$ and
	$$U^n=\{u_1\cdots u_n\, :\, u_i\in U,\ i=1,\ldots, n \}.$$
The smallest such $d$ is called {\it the order of growth} of $G$ and is denoted by~$d(G)$.
It can be shown that the order of growth of $G$ does not depend on the choice of symmetric generating set $U$, i.e. it is a universal constant for $G$. It can happen that a group does not have polynomial growth but $\lambda(U^n)$ grows slower than any exponential function of $n$. In this case, we say that $G$ has an {\it intermediate} growth. We refer the interested reader to \cite{BE1}, \cite{BE2}, \cite{Los}, \cite{Grig}, and \cite{FGL1} for more details on these classes of groups.

In the following theorem, we present assumptions on the function $\rho$ under which all conditions of Theorem \ref{T:diff norm-weighted Lp} will hold for the weight $\om$ defined in \eqref{Eq:weight-lenght function}.

\begin{thm}\label{T:diff weighted LP relation-non-weakly subadditive}
Let $G$ be a compactly generated group, let $\tau$ be a length function on $G$, and let $\om$ be a weight of the form $\om(x)=e^{\rho(\tau(x))}$, where $\rho:\N\cup \{0\}\to \R^+$ is an increasing concave function with $\rho(0)=0$ and $\rho(n)\to \infty$ as $n\to\infty$. Furthermore, let $u:G\to\R^+$  be defined by
\begin{align}\label{Eq:subexpo weight decreasing quotient-2}
u(x)=e^{[\rho(2\tau(x))-2\rho(\tau(x))]}, \quad x\in G.
\end{align}
Suppose that
\begin{align}\label{rho: growth condition}
\ds\limsup_{n\to\infty}\left\{\frac{\rho(2n)}{\rho(n)}\right\}<2\quad\text{and}\quad u\in L^s(G)
\end{align}
for some $0<s<q$. Then there are $0<\theta <1$ and $C>0$ such that for every $f\in L^p(G,\om)\cap L^2(G)$, the relation \eqref{Eq:weighted Lp diff formula} holds.
\end{thm}

\begin{proof}
Since $\rho$ is concave and $\rho(0)=0$, we have that $\rho(2n)\le2\rho(n)$, $n\in\N\cup\{0\}$, and hence $\ds u(x)=e^{[\rho(2\tau(x)-2\rho(\tau(x)))]}\le1$ for every $x\in G$. Also, because $\rho$ is increasing, we have that
$$
\sigma(x)=\om(x)u(x)=e^{\rho(\tau(x))}e^{[\rho(2\tau(x)-2\rho(\tau(x)))]}=e^{[\rho(2\tau(x)-\rho(\tau(x))]}\ge1,\quad x\in G.
$$
Moreover, the very same proof given in \cite[Theorem~2.2]{OSS2} shows that $u$ satisfied \eqref{algebra condition}.
Hence it is only left to verify that \eqref{u_integral condition} holds. Since $u\in L^s(G)$, it suffices to show that there exist $\tilde{s}<q$, $r>0$, and $N\in\mathbb{N}$ such that
$$
u(x)^{\tilde{s}}\om(x)^{r}\le u(x)^s,\quad x\in G\setminus U^N.
$$
Rewriting the above inequality in terms of $\rho$ (using \eqref{Eq:subexpo weight decreasing quotient-2}) and taking natural logarithm of both sides, we get
$$
\tilde{s}(\rho(2\tau(x))-2\rho(\tau(x)))+r\rho(\tau(x))\le s(\rho(2\tau(x))-2\rho(\tau(x))),\quad x\in G\setminus U^N.
$$
Replacing $\tau(x)$ with $n$ for simplicity and rearranging the terms we obtain
$$
(\tilde{s}-s)\rho(2n)\le(2\tilde{s}-2s-r)\rho(n),\quad n\in\N,\ n>N.
$$
Since $s<q$, we can pick $s<\tilde{s}<q$, in which case it would be enough to show the existence of $r>0$ and $N\in\N$ such that
$$
\frac{\rho(2n)}{\rho(n)}\le\frac{2\tilde{s}-2s-r}{\tilde{s}-s}=2-\frac{r}{\tilde{s}-s},\quad n\in\N,\ n>N.
$$
But this follows directly from the assumption \eqref{rho: growth condition} so that the lemma is proved.
\end{proof}

\begin{rem}\label{importance of rho condition}
(i) As it was already mentioned in the proof of Theorem \ref{T:diff weighted LP relation-non-weakly subadditive}, since $\rho$ is concave and $\rho(0)=0$, we always have that $\ds\frac{\rho(2n)}{\rho(n)}\le2$, $n\in\N$. However we need a stronger assumption to obtain the differential norm relation.

(ii) It follows from \cite[Corollary 5.2]{OS1} and \cite[Theorem 2.3]{OSS2} that the conditions \eqref{rho: growth condition} hold for polynomial weights \eqref{Eq:poly weight-defn} with $\beta>d/q$ and all subexponential weights \eqref{Eq:Expo weight-defn} so that Theorem~\ref{T:diff weighted LP relation-non-weakly subadditive} can be applied to these classes of weights. However it does not work for subexponential weights \eqref{Eq:Expo weight II-defn} since they do not satisfy the condition \eqref{rho: growth condition} regarding the growth of $\rho$.

(iii) The condition \eqref{rho: growth condition} is not only sufficient to use the approach of Theorem \ref{T:diff norm-weighted Lp} but in most natural cases it is also necessary for the inequality
$$
\|f\ast f\|_{p,\om}\le C\|f\|_{p,\om}^{1+\theta}\|f\|_2^{1-\theta},\quad f\in\ell^p(G,\om),
$$
to be satisfied. For example, if we take $G=\mathbb{Z}$, $\tau(n)=|n|$, and substitute $f=\delta_n$, $n\in\N$, in the preceding inequality, then we get
$$
\om(2n)\le C \big(\om(n)\big)^{1+\theta},\quad n\in\N.
$$
In terms of $\rho$, the last inequality is equivalent to
$$
\frac{\rho(2n)}{\rho(n)}\le 1+\theta+\frac{\ln C}{\rho(n)},\quad n\in\N,
$$
which implies \eqref{rho: growth condition} since $\theta<1$ and $\rho(n)\to \infty$ as $n\to\infty$.
\end{rem}

\section{Norm-controlled inversion of $\ell^p(G,\omega)$ in $C^*_r(G)$}\label{S:Norm control-weighted Lp sapce-discrete groups}

We are now ready to present our main results on discrete groups.

\begin{thm}\label{T:norm control general-discrete}
Let $G$ be a discrete group, let $\om$ be a weight on $G$, and let $1\le p<\infty$ and $q$ be the index conjugate to $p$. Suppose that there exists a bounded function $u:G\to\R^+$,  $s<q$, and $r>0$ such that $\sigma=\om u$ is bounded away from zero, 
\begin{align*}
\frac{\om(xy)}{\om(x)\om(y)}\le u(x)+u(y),\quad x,y\in G,\quad\text{and}\quad
\sum_{x\in G} u(x)^s\om(x)^{r}dx<\infty.
\end{align*}
Then $\ell^p(G,\omega)$ is a Banach $*$-algebra with respect to the convolution product which is inverse closed in $C^*_r(G)$ and $\ell^p(G,\omega)$ admits a norm-controlled inversion in $C^*_r(G)$. This, in particular, holds if $\om$ is weakly subadditive and $\om^{-1}\in \ell^s(G)$ for some $0<s<q$.
\end{thm}

\begin{proof}
We first note that our hypotheses together with Theorem \ref{Lp algebra} imply that $\ell^p(G,\omega)$ is a Banach algebra with respect to convolution and $\ell^p(G,\om)\subseteq C^*_r(G)$. Moreover, by Theorem \ref{T:diff norm-weighted Lp},
there exist $0<\theta<1$ and $C>0$ such that
\begin{align}\label{differential norm-app}
\|f\ast f\|_{p,\omega}\le C\|f\|_{p,\omega}^{1+\theta}\|f\|_2^{1-\theta},\quad f\in\ell^p(G,\omega).
\end{align}
Finally, since $G$ is discrete, we have that $C^*_r(G)\subset \ell^2(G)$ so that
\begin{align}\label{cv2 norm1}
\|f\|_2\le\|f\|_{C^*_r},\quad f\in C^*_r(G).
\end{align}
Hence, if we combine inequalities \eqref{differential norm-app} and \eqref{cv2 norm1}, the result will follow from Proposition \ref{norm control general}.
\end{proof}

In the following corollary, we summarize various cases that our methods can be applied to obtain the norm-control inversion for weighted $\ell^p$ spaces.


\begin{cor}\label{C:norm control general-discrete-examples}
Let $G$ be a discrete group, let $\om$ be a weight on $G$, and let $1\le p<\infty$ and $q$ be the index conjugate to $p$. Then $\ell^p(G,\omega)$ admits a norm-controlled inversion in $C^*_r(G)$ in either of the following cases:\\
$(i)$ $G$ is locally finite and $\om$ is the weight \eqref{Eq:weight-locally finite} with $\om^{-1}\in \ell^s(G)$ for some $0<s<q$;\\
$(ii)$ $G$ is a finitely generated group of polynomial growth and $\om:=\om_\beta$ is the weight \eqref{Eq:poly weight-defn} with $\beta>d(G)/q$;\\
$(iii)$ $G$ is a finitely generated group of polynomial growth and $\om:=\om_{\alpha,C}$ is the weight \eqref{Eq:Expo weight-defn};\\
$(iv)$ $G$ is a finitely generated group of intermediate growth whose growth is bounded by $e^{n^{\alpha_0}}$ and $\om:=\om_{\alpha,C}$ is the weight \eqref{Eq:Expo weight-defn} with $0<\alpha_0<\alpha<1$.
\end{cor}

\begin{proof}
As was mentioned in Example~\ref{E:locally finite-diff Lp norm relation}, the weight \eqref{Eq:weight-locally finite} is weakly subadditive and hence (i) follows from Theorem \ref{T:diff weighted LP relation-non-weakly subadditive}. Statements (ii) and (iii) also follow from Theorem \ref{T:diff weighted LP relation-non-weakly subadditive} and Remark~\ref{importance of rho condition}(ii). So we just need to prove (iv). Again, according to Theorem~\ref{T:diff weighted LP relation-non-weakly subadditive}, it suffices to verify that
$$
u(x)=e^{[C(2\tau(x))^{\alpha}-2C\tau(x)^{\alpha}]}=e^{-C(2-2^{\alpha})\tau(x)^{\alpha}}\in \ell^s(G)
$$
for some $0<s<q$. However this is true for any $s>0$. Indeed,
\begin{align*}
\sum_{x\in G} u(x)^s=1+\sum\limits_{n=1}^{\infty}|U^n\setminus U^{n-1}| \cdot{e^{-sC(2-2^{\alpha})n^{\alpha}}}\le 1+\sum\limits_{n=1}^{\infty} {e^{n^{\alpha_0}-sC(2-2^{\alpha})n^{\alpha}}}<\infty,
\end{align*}
since $0<\alpha_0<\alpha<1$, and the corollary is proved.
\end{proof}

We would like to point out that intermediate groups satisfying conditions of Corollary \ref{C:norm control general-discrete-examples}(iv) are shown to exist, see for example \cite{BE1} and \cite{BE2}.

\section{Some generalizations for algebras on a non-discrete group $G$}\label{S:Norm control-weighted Lp sapce-Non discrete groups}

The main difference between the cases when the group $G$ is discrete and when it is not discrete is in the existence of the unit in the algebra $L^p(G,\omega)$. There are two approaches that allow us to consider inverse closedness or norm controlled inversion for an algebra without unit: switching to quasi-inverses or adjoining the unit to our algebra. We first consider the quasi-inverse closedness of the algebra $L^p(G,\om)\cap L^2(G)$ in $B(L^2(G))$. Let us start by reminding the corresponding definitions which can be found, for example, in \cite{BD}.

\begin{defn}
Let $\A$ be a Banach algebra and $a,b\in \A$. The {\em quasi-product} $a\circ b$ of $a$ and $b$ in $\A$ is then defined by
\begin{align}\label{quasi-product}
a\circ b=a+b-ab.
\end{align}
\end{defn}
The operation of quasi-product is associative and satisfies
$$
a\circ 0=0\circ a=a,\quad a\in\A,
$$
which means that the zero element plays a role of quasi-unit in $\A$. If $\A$ is a $*$-algebra, then we also naturally have that
$$
(a\circ b)^*=b^*\circ a^*, \quad a,b\in\A.
$$
However, the distributive and constant multiple rules for quasi-products are different from the usual ones. But since we are not going to use them, we do not state the precise formulations here.

\begin{defn}
Let $\A$ be a Banach algebra and $a\in\A$. An element $a^0\in\A$ is called a {\em quasi-inverse} for $a$ if
$$
a\circ a^0=a^0\circ a=0.
$$
If an element $a\in\A$ has a quasi-inverse, then it is called {\em quasi-invertible} or {\em quasi-regular} and otherwise it is called {\em quasi-singular}. The sets of all quasi-invertible and quasi-singular elements of $A$ are denoted by $\mathrm{q-Inv}(\A)$ and $\mathrm{q-Sing}(\A)$ respectively.
\end{defn}

The motivation for defining quasi-products and quasi-inverses comes from the desire to define the spectrum of an element of a non-unital algebra $\A$, and the definition (\ref{quasi-product}) follows from the relation
\begin{align}\label{4}
e-a\circ b=(e-a)(e-b), \quad a,b\in\A,
\end{align}
that holds in the case when $\A$ has a unit element $e$. In particular, we see that $a\in\A$ is quasi-invertible if and only if $e-a$ is invertible and $a^0=e-(e-a)^{-1}$. The relation \eqref{4} also justifies the following definition of the spectrum in a non-unital Banach algebra.

\begin{defn}
Let $A$ be a Banach algebra without a unit and $a\in\A$. We define the {\em spectrum} $\mathrm{Sp}_{\A}(a)$ of $a$ in $\A$ by
\begin{align*}
\sigma_{\A}(a)=\{0\}\cup\left\{\lambda\in\mathbb{C}\setminus\{0\}\,:\,\frac1{\lambda}\,a\in\mathrm{q-Sing(A)}\right\}.
\end{align*}
\end{defn}

The notion of spectrum is closely related to (quasi-)inverse closedness. More precisely, if an algebra $\A$ is contained in an algebra $\B$ then $\A$ is (quasi-)inverse closed in $\B$ if and only if $\sigma_{\A}(a)=\sigma_{\B}(a)$, $a\in \A$. (Of course, here we mean that $\A$ is quasi-inverse closed in $\B$ if for every $a\in\A$ the existence of $a^0\in \B$ implies that $a^0\in\A$.) Thanks to the generalization of a result of Hulanicki (\cite[Proposition~2.5]{Hulanicki}) given in \cite[Lemma 3.1]{FGL1}, when $\A$ and $\B$ have common involution and $\B$ is symmetric, the equality of the spectra $\sigma_{\A}(a)=\sigma_{\B}(a)$ is equivalent to the equality of the spectral radii $\rho_{\A}(a)=\rho_{\B}(a)$, $a\in\A$. This allows us to prove the quasi-inverse closedness of the algebra $\A=L^p(G,\om)\cap L^2(G)$ in $B(L^2(G))$ under similar assumptions as in Theorem~\ref{T:norm control general-discrete} for a non-necessary discrete group $G$.

\begin{prop}\label{quasi-inverse closedness}
Let $G$ be a locally compact group, let $\om\ge1$ be a weight on $G$, and let $1\le p<\infty$ and $q$ the index conjugate to $p$. Suppose that there exists a bounded function $u:G\to\R^+$,  $s<q$, and $r>0$ such that $\sigma(x)=\om(x)u(x)$, $x\in G$, is bounded away from zero,
\begin{align*}
\frac{\om(xy)}{\om(x)\om(y)}\le u(x)+u(y),\quad x,y\in G,\quad\text{and}
\end{align*}
\begin{align*}
\int\limits_G u(x)^s\om(x)^{r}dx<\infty.
\end{align*}
Then $\A=L^p(G,\omega)\cap L^2(G)$ with the norm
$$
\|f\|_{\A}=\max\{\|f\|_{p,\om}\,,\|f\|_2\},\quad f\in\A,
$$
is a Banach $*$-algebra under convolution and it is quasi-inverse closed in $B(L^2(G))$.
\end{prop}

\begin{proof}
We first show that $\A$ is a Banach algebra under convolution. Since by Theorem~\ref{Lp algebra} under our assumptions $L^p(G,\om)$ is a Banach algebra, we have that for some $C_1>0$
\begin{align*}
\|f\ast g\|_{p,\om}\le C_1 \|f\|_{p,\om}\|g\|_{p,\om}, \quad f,g\in L^p(G,\om).
\end{align*}
Our assumptions also imply that $L^p(G,\om)\subset L^1(G)$ so that it follows that for some $C_2>0$ we have
\begin{align}\label{2}
\|f\ast g\|_2\le\|f\|_1\|g\|_2\le C_2 \|f\|_{p,\om}\|g\|_2, \quad f\in L^p(G,\om),\, g\in L^2(G).
\end{align}
Therefore, for $f,g\in\A=L^p(G,\om)\cap L^2(G)$ we have that
\begin{align*}
\|f\ast g\|_{\A}&=\max\{\|f\ast g\|_{p,\om}\,,\|f\ast g\|_2\}\le\max\{C_1,C_2\}\|f\|_{p,\om}\max\{\|g\|_{p,\om}\,,\|g\|_2\}\\&\le\max\{C_1,C_2\}\|f\|_{\A}\|g\|_{\A},
\end{align*}
which means that $\A$ is indeed a Banach algebra under convolution. It also follows from \eqref{2} that $\A$ continuously embeds in $B(L^2(G))$.

As was discussed above, to prove that $\A$ is quasi-inverse closed in $B(L^2(G))$ it is enough to show that $\rho_{\A}(f)=\rho_{B(L^2(G))}(f)$, $f\in\A$. By \eqref{Eq:weighted Lp diff formula}, there is $C>0$ and $0<\theta<1$ such that
\begin{align*}
\|f\ast f\|_{p,\om}\le C\|f\|_{p,\om}^{1+\theta}\|f\|_2^{1-\theta},\quad f\in\A.
\end{align*}
Since for any $f\in\A$ we have that $\|f\|_{\A}\ge\|f\|_2$, $\|f\|_{\A}\ge\|f\|_{p,\om}$ and $\theta>0$, we can combine the above inequality with \eqref{2} for $g=f$ to obtain
\begin{align}\nonumber
\|f\ast f\|_{\A}&=\max\{\|f\ast f\|_{p,\om}\,,\|f\ast f\|_2\}\le\max\{C\|f\|_{p,\om}^{1+\theta}\|f\|_2^{1-\theta},C_2(\|f\|_{p,\om}\|f\|_2^{\theta})\|f\|_2^{1-\theta}\}\\ \label{3} &\le\max\{C,C_2\}\|f\|_{\A}^{1+\theta}\|f\|_2^{1-\theta},\quad f\in\A.
\end{align}
For $n\in\mathbb{N}$, we denote the $n$-{th} convolution power of a function $g$ by $g^{(n)}$. We then apply \eqref{3} to $f=g^{(n)}$, $g\in\A$, $n\in\N$, to get
\begin{align*}
\|g^{(2n)}\|_{\A}\le \max\{C,C_2\}\|g^{(n)}\|_{\A}^{1+\theta}\|g^{(n)}\|_2^{1-\theta}\le \max\{C,C_2\}\|g^{(n)}\|_{\A}^{1+\theta}\|g^{(n-1)}\|_{B(L^2(G))}^{1-\theta}\|g\|_2^{1-\theta}.
\end{align*}
As in the proof of Proposition \ref{norm control general}, we can now take the $n$-th root and $\lim_{n\to\infty}$ of both sides of the above inequality to obtain that $(\rho_{\A}(g))^2\le(\rho_{\A}(g))^{1+\theta}(\rho_{B(L^2(G))}(g))^{1-\theta}$, which implies that $\rho_{\A}(g)\le\rho_{B(L^2(G))}(g)$, $g\in\A$. Because the reverse inequality follows immediately from the embedding $\A\subseteq B(L^2(G))$, we get the desired equality of spectral radii $\rho_{\A}(g)=\rho_{B(L^2(G))}(g)$, $g\in\A$, and the theorem is proved.
\end{proof}

We now turn to the second approach of unitization of a non-unital Banach algebra which will allow us not only to talk about inverse closedness but also consider the norm controlled inversion.

\begin{defn}\label{unitization}
Let $\A$ be a Banach algebra without unit. The {\em unitization} $\widetilde{\A}$ of $\A$ is the set $\A\times\C$ with the following operations of addition, scalar multiplication, and product ($a,b\in\A,\ t,s\in\C$):
\begin{align*}
(a,t)+(b,s)&=(a+b,t+s),\\
s(a,t)&=(sa,st),\\
(a,t)(b,s)&=(ab+sa+tb,ts),
\end{align*}
and with the norm
$$
\|(a,t)\|=\|a\|+|t|.
$$
In the case when $\A$ is a $*$-algebra, the involution can be extended from $\A$ to $\Au$ by
$$
(a,t)^*=(a^*,\overline{t}).
$$
\end{defn}

It is easy to see that $\Au$ is a Banach algebra with the unit $(0,1)$ and that the mapping $a\mapsto (a,0)$ establishes an isometric ($*$-)isomorphism of $\A$ into $\Au$. Moreover, by \cite[Lemma 5.2]{BD}, the following relation between the spectra in $\A$ and $\Au$ holds.

\begin{prop}\label{spectral equivalence}
Let $A$ be a non-unital Banach algebra, and let $\Au$ be its unitization. Then
\begin{align}\label{spectra}
\sigma_{\A}(a)=\sigma_{\Au}\bigl((a,0)\bigr),\quad a\in \A.
\end{align}
\end{prop}

From this we can obtain the following corollary of Theorem~\ref{quasi-inverse closedness}.

\begin{cor}\label{inverse closed-unitization}
Let $G$ be a non-discrete locally compact group, let $\om$ be a weight on $G$, and let $1\le p<\infty$ satisfy the conditions of  Proposition~\ref{quasi-inverse closedness}. Let $\A=L^p(G,\om)\cap L^2(G)$ and $\Au$ be its unitization with the convolution product extended as in Definition~\ref{unitization}. Then $\Au$ is continuously embedded into $B(L^2(G))$ and it is inverse closed in $B(L^2(G))$.
\end{cor}

\begin{proof}
We map $\Au$ into $B(L^2(G))$ in a natural way by letting
$$
(f,t)(g)=f\ast g+tg,\quad f\in\Au,\,g\in L^2(G),
$$
which comes from the multiplication of $(f,t)$ and $(g,0)$ in $\Au$. It is clear that since $G$ is non-discrete, this map is injective. Also the continuity of this embedding follows immediately from the continuity of the embedding of $A$ into $B(L^2(G))$ and the definition of norm on $\Au$:
$$
\|(f,t)(g)\|_2=\|f\ast g +tg\|_2\le\|f\ast g\|_2+|t|\|g\|_2\le\|f\|_A\|g\|_2+|t|\|g\|_2=\|(f,t)\|_{\Au}\|g\|_2.
$$
In order to prove the inverse closedness, it is enough to show that $\sigma_{\Au}\bigl(\tilde{f}\bigr)=\sigma_{B(L^2(G))}\bigl(\tilde{f}\bigr)$, $\tilde{f}\in\Au$. Let $\tilde{f}=(f,t)=(f,0)+t(0,1)$ for some $f\in\A$, $t\in\C$. Then since we already know from the proof of Theorem~\ref{quasi-inverse closedness} that $\sigma_{\A}(f)=\sigma_{B(L^2(G))}(f)$ and $(0,1)$ is the unit in $\Au$, we can apply Proposition~\ref{spectral equivalence} to obtain the desired equality of the spectra:
\begin{align}
\sigma_{\Au}\bigl(\tilde{f}\bigr)=t+\sigma_{\Au}\bigl((f,0)\bigr)=t+\sigma_{\A}(f)=t+\sigma_{B(L^2(G))}(f)= \sigma_{B(L^2(G))}\bigl(\tilde{f}\bigr).
\end{align}
\end{proof}

To deal with the norm controlled inversion we will need the following technical generalization of the inequality \eqref{3}.

\begin{lem}\label{pytlik-unitization}
Let $G$, $\om$, $p$, $\A$, and $\Au$ be as in Corollary~\ref{inverse closed-unitization}. Then there is a constant $D>0$ and $0<\theta<1$ such that
\begin{align}\label{inequality for unitized case}
\bigl\|\tilde{f}\ast\tilde{f}\bigr\|_{\Au}\le D \bigl\|\tilde{f}\bigr\|_{\Au}^{1+\theta}\bigl\|\tilde{f}\bigr\|_{2}^{1-\theta},\quad \tilde{f}\in\Au,
\end{align}
where $\bigl\|\tilde{f}\bigr\|_{2}=\|(f,t)\|_2=\sqrt{\|f\|_2^2+|t|^2}$, $f\in\A$, $t\in\C$.
\end{lem}
\begin{proof}
Let $\tilde{f}=(f,t)$, $f\in\A$, $t\in\C$. Using \eqref{3} and denoting $\widetilde{C}=\max\{C,C_2,1\}$, we have
\begin{align}\nonumber
\bigl\|\tilde{f}\ast\tilde{f}\bigr\|_{\Au}&=\bigl\|(f,t)\ast(f,t)\bigr\|_{\Au}=\|f\ast f+2tf\|_{\A}+|t|^2\le \|f\ast f\|_{\A}+2|t|\|f\|_{\A}+|t|^2\\\nonumber &\le \widetilde{C}(\|f\|_{\A}^{1+\theta}\|f\|_2^{1-\theta}+2|t|\|f\|_{\A}+|t|^2).
\end{align}
Then since $\|(f,t)\|_2\ge\frac1{\sqrt2}(\|f\|_2+|t|)$, to prove \eqref{inequality for unitized case} it is enough to show that for some $\widetilde{D}>0$
\begin{align*}
\|f\|_{\A}^{1+\theta}\|f\|_2^{1-\theta}+2|t|\|f\|_{\A}+|t|^2\le\widetilde{D} (\|f\|_{\A}+|t|)^{1+\theta}(\|f\|_2+|t|)^{1-\theta},\quad f\in\A,\, t\in\C.
\end{align*}
If we let $\alpha=\|f\|_{\A}$, $\beta=\|f\|_2$, and $x=|t|$, then it is enough to show that
\begin{align}\label{5}
F_{\alpha,\beta}(x)=\widetilde{D}(\alpha+x)^{1+\theta}(\beta+x)^{1-\theta}-(\alpha^{1+\theta}\beta^{1-\theta}+2x\alpha+x^2)\ge0,\quad \alpha,\beta,x\ge0.
\end{align}
We now fix $\alpha,\beta\ge0$. The above inequality will be satisfied for $x=0$ provided $\widetilde{D}\ge1$. Hence, \eqref{5} will be proved if we could show that there is $\widetilde{D}\ge1$ such that $F'_{\alpha,\beta}(x)\ge0$, $x>0$, for all $\alpha,\beta\ge0$.
\begin{align*}
F'_{\alpha,\beta}(x)&=\widetilde{D}\left((1+\theta)(\alpha+x)^{\theta}(\beta+x)^{1-\theta}+ (1-\theta)(\alpha+x)^{1+\theta}(\beta+x)^{-\theta}\right)-2\alpha-2x\\&=(\alpha+x)\left(\widetilde{D}\left((1+\theta) \left(\frac{\beta+x}{\alpha+x}\right)^{1-\theta}+(1-\theta) \left(\frac{\alpha+x}{\beta+x}\right)^{\theta}\right)-2\right)\\&=(\alpha+x)\left(\widetilde{D}((1+\theta)\gamma^{1-\theta}+ (1-\theta)\gamma^{-\theta})-2\right),\quad\text{where}\ \gamma=\frac{\beta+x}{\alpha+x}.
\end{align*}
Since $\alpha+x>0$ and $\inf_{\gamma>0}\left( (1+\theta)\gamma^{1-\theta}+(1-\theta)\gamma^{-\theta}\right)>0$,
we can choose $\tilde{D}$ large enough so that $F'_{\alpha,\beta}(x)\ge0$, for all $x>0$, $\alpha,\beta\ge0$.
This completes the proof.
\end{proof}


We will also need the following estimate on the norm of convolution powers of elements of $\Au$.

\begin{lem}\label{convolution power estimate}
Let $G$ be a non-discrete locally compact group. Then for every $n\in\N$ we have
\begin{align}\label{cpe}
\|\tilde{f}^{(n)}\|_2\le n\|\tilde{f}\|_{B(L^2(G))}^{n-1}\|\tilde{f}\|_2,\quad \tilde{f}\in\Au,\ A=L^p(G,\omega)\cap L^2(G).
\end{align}
\end{lem}
\begin{proof}

We first show that if $\tilde{f}=(f,t)$, then $\|\tilde{f}\|_{B(L^2(G))}\ge|t|$. Suppose that  $\left\{g_{U}=\frac{\chi_U}{\lambda(U)}\right\}$ is the standard bounded approximate identity in $L^1(G)$ ($\lambda$ is the Haar measure of $G$). Then, as $U\to \{e\}$, $\|f\ast g_U\|_2\le\|f\|_2\|g_U\|_1\to \|f\|_2$ and $\|g_U\|_2\to \infty$ since $G$ is not discrete. It follows that
\begin{align*}
\|\tilde{f}\|_{B(L^2(G))}\ge\sup_{U}\frac{\|f\ast g_U+tg_U\|_2}{\|g_U\|_2}\ge\lim_{N\to\infty}\frac{N|t|-2\|f\|_2}{N}=|t|,
\end{align*}
and our claim is verified.

We now fix $\tilde{f}=(f,t)\in \Au$ and prove (\ref{cpe}) by induction on $n$. The base for $n=1$ is obvious. Now let the inequality hold for $n=m$. We show that then it also holds for $n=m+1$.
\begin{align*}
\|\tilde{f}^{(m+1)}\|_2&=\|\tilde{f}^{(m)}\ast(f,0)+t\tilde{f}^{(m)}\|_2\le\|\tilde{f}^{(m)}\|_{B(L^2(G))}\|f\|_2+m|t| \|\tilde{f}\|_{B(L^2(G))}^{m-1}\|\tilde{f}\|_2\\ &\le \|\tilde{f}\|^{m}_{B(L^2(G))}\|\tilde{f}\|_2+m\|\tilde{f}\|_{B(L^2(G))}\cdot\|\tilde{f}\|^{m-1}_{B(L^2(G))}\|\tilde{f}\|_2= (m+1)\|\tilde{f}\|^{m}_{B(L^2(G))}\|\tilde{f}\|_2.
\end{align*}
\end{proof}

We are now ready to prove the main result of this section.

\begin{thm}\label{norm control-unitization}
Let $G$ be a non-discrete locally compact group, $\om\ge1$ be a weight on $G$, $1\le p<\infty$, and $q$ be the index conjugate to $p$. Suppose that there exists a bounded function $u:G\to\R^+$,  $s<q$, and $r>0$ such that $\sigma(x)=\om(x)u(x)$, $x\in G$, is bounded away from zero,
\begin{align*}
\frac{\om(xy)}{\om(x)\om(y)}\le u(x)+u(y),\quad x,y\in G,\quad\text{and}
\end{align*}
\begin{align*}
\int\limits_G u(x)^s\om(x)^{r}dx<\infty.
\end{align*}
Further, let $\A=L^p(G,\om)\cap L^2(G)$ and $\Au$ be its unitization. Then $\Au$ is inverse closed in $B(L^2(G))$ and $\Au$ admits a norm-controlled inversion in $B(L^2(G))$.
\end{thm}

\begin{proof}
We already know that $\Au$ is inverse closed in $B(L^2(G))$ as a result of Corollary~\ref{inverse closed-unitization}, so it only remains to prove the existence of a norm control in the inversion. We are going to do that by following the lines of the proof of Proposition ~\ref{norm control general} and making technical adjustments where it is necessary. Let $\tilde{g}\in\Au$ and $k\in\N$. Combining Lemma~\ref{pytlik-unitization} for $\tilde{f}=\tilde{g}^{(2^{k-1})}$ with Lemma~\ref{convolution power estimate} for $n=2^{k-1}$,  we obtain
\begin{align}\label{eq:non-unital inequality}
\bigl\|\tilde{g}^{(2^k)}\bigr\|_{\Au}\le 2^{(k-1)(1-\theta)}D \bigl\|\tilde{g}^{(2^{k-1})} \bigr\|_{\Au}^{1+\theta}\bigl\|\tilde{g}\bigr\|_{B(L^2(G))}^{(2^{k-1}-1)(1-\theta)}\|\tilde{g}\|_2^{1-\theta}.
\end{align}

We now use \eqref{eq:non-unital inequality} instead of \eqref{eq:estimate} to proceed. Let $\beta_n=\frac{\|\tilde{g}^{(n)}\|_{\Au}}{\|\tilde{g}\|^n_{B(L^2(G))}}$.  Then \eqref{eq:non-unital inequality} implies that
\begin{align*}
\beta_{2^k}\le\frac{D\|\tilde{g}\|_2^{1-\theta}}{\|\tilde{g}\|_{\bg}^{1-\theta}}\,2^{(k-1)(1-\theta)}\beta_{2^{k-1}}^{1+\theta}, \quad k\in\N.
\end{align*}
Using induction in $k$, it can be easily shown that
\begin{align}\label{eq:preliminary}
\beta_{2^k}\le\left(\frac{D\|\tilde{g}\|_2^{1-\theta}}{\|\tilde{g}\|_{\bg}^{1-\theta}}\right)^{\frac{(1+\theta)^k-1}{\theta}} 2^{(1-\theta)S_{k}(1+\theta)} \beta_1^{(1+\theta)^k}, \quad k\in\N,
\end{align}
where
$$
S_m(x)=m-1+xS_{m-1}(x),\quad m\in\N,\ \  \text{and}\ \  S_1(x)=0.
$$
It follows that $S_k(x)=\sum_{j=0}^{k-2} (k-1-j)x^j$, and a standard computation yields that
$$
S_k(x)=\frac{x^k-kx+(k-1)}{(x-1)^2},\quad k\in\N.
$$
Hence, \eqref{eq:preliminary} becomes
\begin{align*}
\beta_{2^k}\le\left(\frac{D\|\tilde{g}\|_2^{1-\theta}}{\|\tilde{g}\|_{\bg}^{1-\theta}}\right)^{\frac{(1+\theta)^k-1}{\theta}} 2^{\frac{(1-\theta)((1+\theta)^k-k\theta-1)}{\theta^2}} \beta_1^{(1+\theta)^k}, \quad k\in\N.
\end{align*}
Using the definition of $\beta_n$ and the fact that $\|\tilde{g}\|_2\le\|\tilde{g}\|_{\Au}$, we obtain

\begin{align}\label{eq:power estimate}
\|\tilde{g}^{(2^k)}\|_{\Au}&\le \|\tilde{g}\|_{\bg}^{2^k} \left(\frac{D\|\tilde{g}\|_2^{1-\theta}}{\|\tilde{g}\|_{\bg}^{1-\theta}}\right)^{\frac{(1+\theta)^k-1}{\theta}} 2^{\frac{(1-\theta)((1+\theta)^k-k\theta-1)}{\theta^2}}\left(\frac{\|\tilde{g}\|_{\Au}}{\|\tilde{g}\|_{\bg}}\right)^{(1+\theta)^k}\\ \nonumber &\le D^{\frac{(1+\theta)^k-1}{\theta}}2^{\frac{(1-\theta)((1+\theta)^k-k\theta-1)}{\theta^2}} \|\tilde{g}\|_{\Au}^{\frac{(1+\theta)^k-(1-\theta)}{\theta}} \|\tilde{g}\|_{\bg}^{2^k-\frac{(1+\theta)^k-(1-\theta)}{\theta}},\quad k\in\N.
\end{align}
Repeating the argument with the dyadic expansion of $n$ given in the proof of Proposition ~\ref{norm control general}, we will have that
\begin{align*}
\sum\limits_{n=0}^{\infty} \|\tilde{g}^{(n)}\|_{\Au}&\le \prod\limits_{k=0}^{\infty}\left(1+D^{\frac{(1+\theta)^k-1}{\theta}}2^{\frac{(1-\theta)((1+\theta)^k-k\theta-1)}{\theta^2}} \|\tilde{g}\|_{\Au}^{\frac{(1+\theta)^k-(1-\theta)}{\theta}} \|\tilde{g}\|_{\bg}^{2^k-\frac{(1+\theta)^k-(1-\theta)}{\theta}}\right).
\end{align*}
This infinite product is convergent if and only if
$$
\sum_{k=0}^{\infty} D^{\frac{(1+\theta)^k-1}{\theta}}2^{\frac{(1-\theta)((1+\theta)^k-k\theta-1)}{\theta^2}} \|\tilde{g}\|_{\Au}^{\frac{(1+\theta)^k-(1-\theta)}{\theta}} \|\tilde{g}\|_{\bg}^{2^k-\frac{(1+\theta)^k-(1-\theta)}{\theta}}<\infty,
$$
and since $\theta<1$, this happens exactly when $\|\tilde{g}\|_{\bg}<1$.

Now the rest of the proof including the calculation of the norm controlling function follows from the corresponding part of the proof of Proposition ~\ref{norm control general} since $\bg$ is a $C^*$-algebra.
\end{proof}

We finish this section with the following corollary which provides many cases when norm-controlled inversion happens for non-discrete groups. We omit the proof as it is very similar to the one presented in Corollary \ref{C:norm control general-discrete-examples}.

\begin{cor}\label{C:norm control general-NON discrete-examples}
Let $G$ be a non-discrete locally compact group, $\om\ge1$ be a weight on $G$, $1\le p<\infty$, and $q$ be the index conjugate to $p$. Let $\A=L^p(G,\om)\cap L^2(G)$ and $\Au$ be its unitization. Then $\Au$ admits a norm-controlled inversion in $B(L^2(G))$ in either of the following cases:\\
$(i)$ $G$ and $\om$ are as in Example \ref{E:locally finite-diff Lp norm relation} with $\om^{-1}\in L^s(G)$ for some $0<s<q$;\\
$(ii)$ $G$ is a compactly generated group of polynomial growth and $\om:=\om_\beta$ is the weight \eqref{Eq:poly weight-defn} with $\beta>d(G)/q$;\\
$(iii)$ $G$ is a compactly generated group of polynomial growth and $\om:=\om_{\alpha,C}$ is the weight \eqref{Eq:Expo weight-defn};\\
$(iv)$ $G$ is a compactly generated group of intermediate growth whose growth is bounded by $e^{n^{\alpha_0}}$ and $\om:=\om_{\alpha,C}$ is the weight \eqref{Eq:Expo weight-defn} with $0<\alpha_0<\alpha<1$.
\end{cor}

\section{Appendix}

Here we present the proof of Proposition \ref{formula}. As it was shown in Proposition \ref{norm control general},
$$
\|a^{-1}\|_{\A}\le\frac{\|a\|_{\A}}{\|a\|_{\B}^2}\,\prod_{k=0}^{\infty}\left(1+\left(2\frac{\|a\|_{\A}^2}{\|a\|_{\B}^2}\right)^ {(1+\theta)^k}C^{\frac{(1+\theta)^k-1}{\theta}}\left(1-\frac{1}{\|a\|_{\B}^2\|a^{-1}\|_{\B}^2}\right)^{2^k-(1+\theta)^k}\right).
$$
As in \cite{GK}, we note that because $\A\subset\B$ and $\B$ is a C$^*$-algebra, we have that $\|a\|_{\B}\le\|a\|_{\A}$ and $\|a\|_{\B}^{-1}\le\|a^{-1}\|_{\B}$. Therefore, $\|a\|_{\B}\|a^{-1}\|_{\B}\le\|a\|_{\A}\|a^{-1}\|_{\B}=\nu(a)$ and $\|a\|_{\A}/\|a\|_{\B}\le\nu(a)$, implying that
\begin{align}\label{0}
\|a^{-1}\|_{\A}\le\frac{\|a\|_{\A}}{\|a\|_{\B}^2}\,\prod_{k=0}^{\infty}\left(1+\left(2\nu^2(a)\right)^ {(1+\theta)^k}C^{\frac{(1+\theta)^k-1}{\theta}}\left(1-\frac{1}{\nu^2(a)}\right)^{2^k-(1+\theta)^k}\right).
\end{align}
The infinite product on the right can be rewritten as
\begin{align}\label{6}
\prod_{k=0}^{\infty}\left(1+C^{-1/\theta}u^{(1+\theta)^k}v^{2^k}\right)
\end{align}
where $v=1-\frac{1}{\nu^2(a)}$, $u=2C^{1/\theta}v^{-1}(1-v)^{-1}$, and our goal will be to estimate it. Since $\nu(a)\ge1$ and $C\ge1$, we have that $0<v<1$ and $u\ge\frac2{v(1-v)}\ge8>1$.
Let
$$
f(k)=\ln \left(u^{(1+\theta)^k}v^{2^k}\right)=(1+\theta)^k\ln u+2^k\ln v, \quad k\in\R^+.
$$
Using differentiation, we find that $f$ attains its maximum at $k_m=\log_{2/({1+\theta})}{\left(\frac{\ln{(1+\theta)}\cdot\ln u}{\ln2\cdot \ln(v^{-1})}\right)}$. Hence, any term in the infinite product \eqref{6} does not exceed
\begin{align}\label{7}
1+C^{-1/\theta}u^{(1+\theta)^{k_m}}v^{2^{k_m}}=1+C^{-1/\theta}u^{\left(\frac{\ln{(1+\theta)}\cdot\ln u}{\ln2\cdot \ln(v^{-1})}\right)^{\log_{2/({1+\theta})}(1+\theta)}}v^{\left(\frac{\ln{(1+\theta)}\cdot\ln u}{\ln2\cdot \ln(v^{-1})}\right)^{\log_{2/({1+\theta})}2}}.
\end{align}
Let
$$
\gamma=\frac{\ln(1+\theta)}{\ln 2}=\log_2(1+\theta).
$$
Then since $0<\theta<1$, we have that $0<\gamma<1$. It is also easy to see that $\log_{2/({1+\theta})}2=\frac{1}{1-\gamma}$ and $\log_{2/({1+\theta})}(1+\theta)=\frac{\gamma}{1-\gamma}$. This allows us to rewrite the expression in \eqref{7} as follows:
\begin{align}\nonumber
&1+C^{-1/\theta}e^{\ln u\cdot\left(\gamma\frac{\ln u}{\ln(v^{-1})}\right)^\frac{\gamma}{1-\gamma}-\ln {(v^{-1})}\cdot\left(\gamma\frac{\ln u}{\ln(v^{-1})}\right)^\frac{1}{1-\gamma}}\\\label{8}&=1+C^{-1/\theta}e^{(1-\gamma)\gamma^{\frac{\gamma}{1-\gamma}}(\ln u)^\frac{1}{1-\gamma}(\ln {(v^{-1}}))^\frac{-\gamma}{1-\gamma}}=1+C^{-1/\theta} e^{\tilde{C}(\ln u)^\frac{1}{1-\gamma}(\ln {(v^{-1}}))^\frac{-\gamma}{1-\gamma}}.
\end{align}

To estimate the infinite product in \eqref{6}, we split it into two parts: a finite part from $0$ to some $N$ and an infinite part from $N+1$ to infinity. To choose $N$, we first find $k_0\in\R^+$ such that $u^{(1+\theta)^{k_0}}v^{2^{k_0}}=1$, i.e. such that $f(k_0)=0$:
$$
k_0=\log_{2/({1+\theta})}{\left(\frac{\ln u}{\ln(v^{-1})}\right)},
$$
 and then take $N=\lfloor k_0\rfloor$, where $\lfloor\cdot\rfloor$ is a floor function. We now estimate the infinite part of~\eqref{6}. Since $C\ge1$, we have that

\begin{align}\label{9}
\prod_{k=N+1}^{\infty}\left(1+C^{-1/\theta}u^{(1+\theta)^k}v^{2^k}\right)\le \prod_{k=N+1}^{\infty}\left(1+u^{(1+\theta)^k}v^{2^k}\right)\le e^{\,\sum_{k=N+1}^{\infty}u^{(1+\theta)^k}v^{2^k}}.
\end{align}

Let $a_k=u^{(1+\theta)^k}v^{2^k}$ and $\alpha_k=\frac{a_{k+1}}{a_k}=u^{\theta(1+\theta)^k}v^{2^k}$. Then by the choice of $k_0$ and $N$ we have that $a_{k}<1$, $k\ge N+1$. Also, since $u>1$ and $0<\theta<1$,
$$
\frac{\alpha_{k+1}}{\alpha_k}=u^{\theta^2(1+\theta)^k}v^{2^k}<u^{(1+\theta)^k}v^{2^k}=a_k<1,\quad k\ge N+1,
$$
which implies that $a_{N+1+m}\le\alpha_{N+1}^m$, $m\in\N\cup\{0\}$. It follows that
\begin{align}\label{10}
\sum_{k=N+1}^{\infty}u^{(1+\theta)^k}v^{2^k}=\sum_{m=0}^{\infty} a_{N+1+m}\le\frac1{1-\alpha_{N+1}}.
\end{align}
By the choice of $k_0$ and $N$ and using that $u\ge8$, we obtain
$$
\alpha_{N+1}\le\alpha_{k_0}=u^{\theta(1+\theta)^{k_0}}v^{2^{k_0}}=u^{-(1-\theta)(1+\theta)^{k_0}}\le u ^{-(1-\theta)}\le8^{-(1-\theta)}.
$$
Combining the last estimate with \eqref{9} and \eqref{10}, we see that the infinite part of \eqref{6} doesn't exceed a constant $e^{\frac1{1-8^{-(1-\theta)}}}$.

For the finite part of \eqref{6} we have an obvious estimate that follows from \eqref{7} and \eqref{8}:
$$
\prod_{k=0}^{N}\left(1+C^{-1/\theta}u^{(1+\theta)^k}v^{2^k}\right)\le \left(1+C^{-1/\theta} e^{\tilde{C}(\ln u)^\frac{1}{1-\gamma}(\ln {(v^{-1}}))^\frac{-\gamma}{1-\gamma}}\right)^{N+1}.
$$
Since $u\ge8$, $u=2C^{1/\theta}v^{-1}(1-v)^{-1}\ge v^{-1}$, and $N\le k_0$, there exists a constant $C_1$ such that
\begin{align}\nonumber
\prod_{k=0}^{\infty}\left(1+C^{-1/\theta}u^{(1+\theta)^k}v^{2^k}\right)&\le C_1e^{\tilde{C}(\ln u)^\frac{1}{1-\gamma}(\ln {(v^{-1}}))^\frac{-\gamma}{1-\gamma}\cdot{\left(\log_{2/({1+\theta})}{\left(\frac{\ln u}{\ln(v^{-1})}\right)}+1\right)}}\\\label{11}&=C_1e^{\tilde{C}\left(\frac{\ln u}{\ln(v^{-1})}\right)^{\frac1{1-\gamma}}\cdot(\ln {(v^{-1}}))\cdot{\left(\log_{2/({1+\theta})}{\left(\frac{\ln u}{\ln(v^{-1})}\right)}+1\right)}}.
\end{align}

Recalling that $v=1-\frac{1}{\nu^2(a)}$, we obtain
$$
\ln(v^{-1})=\ln\left(1+\frac1{\nu^2(a)-1}\right)\le\frac1{\nu^2(a)-1}\le\frac4{3\nu^2(a)},\quad\text{if}\ \nu(a)\ge2.
$$
We also have that
\begin{align*}
\frac{\ln u}{\ln(v^{-1})}=\frac{\ln(2C^{1/\theta}v^{-1}(1-v)^{-1})}{\ln(v^{-1})}= 1+\frac{\ln(2C^{1/\theta}\nu^2(a))}{\ln\left(1+\frac1{\nu^2(a)-1}\right)}\le \hat{C}\ln(\nu(a))\nu^2(a),\quad\text{if}\ \nu(a)\ge2,
\end{align*}
for some constant $\hat{C}>0$. Combining the above estimates with \eqref{11}, we see that there exists a constant $C_2>0$ such that
\begin{align*}
\prod_{k=0}^{\infty}\left(1+C^{-1/\theta}u^{(1+\theta)^k}v^{2^k}\right)\le C_1e^{C_2(\nu(a))^{\frac{2\gamma}{1-\gamma}}\cdot(\ln {(\nu(a))})^{\frac{2-\gamma}{1-\gamma}}},\quad\text{if}\ \nu(a)\ge2.
\end{align*}
Finally, we combine this with \eqref{0}, and the proof of Proposition \ref{formula} is complete.

\end{document}